\documentclass[10pt,leqno]{article}
\usepackage{epigraph}
\usepackage{amssymb}
\usepackage{amsmath}
\usepackage{tabularx}
\usepackage[arrow,matrix,tips,frame,color,line,poly]{xy}
\usepackage{theorem}
\newtheorem{theorem}[equation]{Theorem}
\newtheorem{corollary}[equation]{Corollary}
\newtheorem{question}[equation]{Question}
\newtheorem{definition}[equation]{Definition}

\newtheorem{lemma}[equation]{Lemma}

\newtheorem{proposition}[equation]{Proposition}
{\theorembodyfont{\rmfamily}

\newtheorem{example}[equation]{Example}
}
\newcommand{\qed}{\hfill $\square$ \medskip}
\newcommand\bx{\boxtimes}
\newenvironment{proof}[1][Proof]{\noindent\textbf{#1.} }{\
\qed}

\newcommand{\Cat}{\mathcal{C}}

\newcommand{\Ob}{Ob(\Cat)}
\newcommand{\ObG}{Ob(\Cat_G)}
\newcommand{\ObH}{Ob(\Cat_H)}
\newcommand{\ObK}{Ob(\Cat_K)}
\newcommand{\ObM}{Ob(\Cat_M)}

\newcommand{\R}{\mathbb R}

\newcommand{\C}{\mathbb C}
\newcommand{\Z}{\mathbb Z}

\renewcommand{\S}{\mathcal S}

\renewcommand{\O}{\mathrm{O}}
\newcommand{\Sp}{\mathrm{Sp}}
\newcommand{\Mp}{\widetilde{\mathrm{Sp}}}

\DeclareMathOperator{\GL}{GL}

\newcommand{\Ext}{\mathrm{Ext}}
\newcommand{\EP}{\mathrm{EP}}

\newcommand\inv{^{-1}}

\newcommand{\frn}{{\mathfrak n}}

\newcommand{\ind}{\mathrm{ind}}

\DeclareMathOperator{\Hom}{Hom}
\DeclareMathOperator{\Ad}{Ad}
\DeclareMathOperator{\Ker}{Ker}

\DeclareMathOperator{\Ind}{Ind}
\DeclareMathOperator{\Res}{Res}

\renewcommand{\sec}[1]{\section{#1}
\renewcommand{\theequation}{\thesection.\arabic{equation}}
  \setcounter{equation}{0}}
\newcommand{\subsec}[1]{\subsection{#1}
\renewcommand{\theequation}{\thesection.\arabic{equation}}}

\begin{document}

\date{}

\title{Euler Poincare Characteristic for the Oscillator Representation}
\author{Jeffrey D. Adams\thanks{The first author was supported in part
    by NSF grant DMS-1317523.} \\Department of
  Mathematics \\ University of Maryland\\College Park, MD 20742
\and Dipendra Prasad \thanks{The second author held Jean-Morlet Chaire at CIRM, Marseille during part of this work.}\\
Tata Institute of Fundamental Research\\
Colaba, Mumbai-400005, India\\
\and Gordan Savin\thanks{The third author was supported in part by NSF grant DMS-1359774.}\\Department of Mathematics\\University of Utah\\Salt Lake City, UT 84112}
\maketitle

{\renewcommand{\thefootnote}{} 
\footnote{
\footnote{The first author was supported in part by  National Science
Foundation Grant \#DMS-1317523}
\footnote{The second author was supported in part by  National Science
Foundation Grant \#DMS-3.14159}
}
\renewcommand{\epigraphflush}{flushleft}
\setlength{\epigraphwidth}{3.2in}
\setlength{\afterepigraphskip}{0pt}
\setlength{\beforeepigraphskip}{0pt}
\epigraph{Dedicated to Roger Howe on his $70^{th}$ birthday}{}

\sec{Introduction}
\label{s:introduction}

Consider a dual reductive pair of subgroups $(G,G')$ of the
metaplectic group $\Mp(2n)$, a two fold cover of $\Sp(2n)=\Sp(2n,F)$ for $F$ a local field.  Let $\omega$ be the oscillator
representation of $\Mp(2n)$. The dual pair correspondence, 
due to Roger Howe, is a bijection between
subsets of the  duals of $G,G'$; we say
$\pi\leftrightarrow \pi'$ if $\Hom_{(G,G')}(\omega,\pi\boxtimes\pi')\ne 0$.
This is a deep theorem, first proved by Howe over $\R$ 
\cite{howe_transcending}, and for $p$-adic fields by Waldspurger
\cite{waldspurger_howe_conjecture} and Gan-Takeda \cite{gan_takeda_howe_conjecture}.

This correspondence plays an important role in the theory of
automorphic forms. It is typically subtle and difficult to compute 
explicitly, and there is a wealth of literature on the subject.
For example see \cite{mvw}.

It is natural to try to simplify the problem by generalizing: instead
of  $\Hom_{(G,G')}(\omega,\pi\boxtimes\pi')$, one considers
$\Ext^i_{(G,G')}(\omega,\pi\boxtimes\pi')$. This suggests
the possibility of studying the Euler-Poincare characteristic
$$\EP(\omega,\pi\boxtimes\pi')=\sum_i(-1)^i\Ext^i_{(G,G')}(\omega,\pi\boxtimes\pi').
$$
There are a number of technical issues to overcome in order to carry
this out, the first of which is to show that
$\Ext^i_{(G,G')}(\omega,\pi\boxtimes\pi')$ is finite dimensional, and
$0$ for sufficiently large $i$. Some of the general theory required can be found in \cite{prasadext}. 

In the theory of dual pairs it is fruitful to consider the question
from a less symmetric point of view.  Fix an irreducible
representation $\pi$ of $G$, and consider the maximal $\pi$-isotypic
quotient of $\omega$. As a $G\times G'$-module, this quotient is
isomorphic to $\pi\boxtimes \Theta(\pi)$ for a smooth $G'$-module
$\Theta(\pi)$, whose (algebraic) dual  is $\Hom_G(\omega,\pi)$.  The main
step in the proof of the duality correspondence is to show that $\Theta(\pi)$
is a finite length $G'$-module, with unique irreducible quotient
$\theta(\pi)$.  Then $\pi\leftrightarrow \theta(\pi)$ is the dual pair
correspondence. As the discussion above suggests, the fine structure
of $\Theta(\pi)$ is also of some interest.

So in our setting we consider the spaces $\Ext^i_G(\omega,\pi)$ as
$G'$-modules. We specialize now to the $p$-adic case.
In order to stay in the category of smooth
representations, it is better to take the smooth vectors
$\Ext^i_G(\omega,\pi)^\infty$.  We would like to know that
$\EP(\omega,\pi)^\infty:=\sum_i (-1)^i\Ext^i_G(\omega,\pi)^\infty$ is a
well-defined element of the Grothendieck group of finite length
representations of $G'$.

For simplicity of exposition, we will only consider the case of type
II dual pairs $(G,G')= (\GL(m),\GL(n))$ in this Introduction,
directing the reader to the main body of the paper for type I dual
pairs.  Fix an irreducible representation $\pi$ of $\GL(m)$. A
preliminary result is that $\EP_G(\omega,\pi)^\infty$ is well defined.

\begin{proposition}
$\Ext^i_G(\omega,\pi)^\infty$ is a finite length module for $G'$ for all $i$, and equal to $0$  if $i>\text{rank}(G)$. 
Hence  $\EP(\omega,\pi)^\infty$ is a well defined element of the Grothendieck group.
\end{proposition}
See Proposition \ref{p:epomega}.  When $i=0$ this is  the well
known fact that $\Theta(\pi)$ has finite length. See Proposition
\ref{p:bigtheta}.

Now suppose $m\le n$ and that $P$ is a parabolic subgroup of $\GL(n)$ with Levi factor $\GL(m)\times \GL(n-m)$. 
It is well known that for $\pi$ an irreducible  representation of $\GL(m)$, 
\begin{equation}
\label{e:ind}
\Hom_{\GL(m) \times \GL(n)}(\omega,\pi \boxtimes i_P^G(\pi\boxtimes 1))\ne 0
\end{equation}
where $i_P^G$ denotes normalized smooth induction from $P$ to $G$. See \cite{mvw}. 

Hence a  naive guess for the explicit dual correspondence for type II
dual pairs would be 
 that for $n \geq m$ (which we can assume without loss of generality since $\GL(m)$ and $\GL(n)$ play a symmetrical role),
the map in \eqref{e:ind} is surjective and realizes the maximal $\pi$-isotypic quotient of $\omega$, and that  the induced
representation has a unique irreducible quotient. If this is the
case, then 
$\Theta(\pi)=i_P^{\GL(n)}(\pi\boxtimes 1)$, and $\theta(\pi)$ is the
unique irreducible quotient of this induced representation.
Generically, of course, the induced representation is irreducible and
this is true. However, in general,  the induced representation may be reducible, 
and computing the explicit dual pair
correspondence amounts to understanding the image of the map  in \eqref{e:ind}, and the
structure of the induced representation.  See
\cite[conjecture on bottom of page 64]{mvw} and \cite{minguez_type_II}.

Our first   main result is that the corresponding naive guess does in
fact hold if one replaces $\Hom_G(\omega,\pi)$ with
$\EP_G(\omega,\pi)^\infty$. 

\begin{theorem}
\label{t:epgln_introduction}
Consider the oscillator representation  $\omega$ for the dual
pair $(\GL(m)$, $\GL(n))$. 
Let $\pi$ be an irreducible representation of $\GL(m)$. 
Then
$$
\EP_{\GL(m)}(\omega,\pi)^{\infty}=
\begin{cases}
0&n<m\\ 
i_{P}^{\GL(n)}(\pi\boxtimes1)&n\ge m
\end{cases}
$$
where $P$ is a parabolic subgroup of $\GL(n)$ with  Levi subgroup $\GL(m)\times \GL(n-m)$.
\end{theorem}
We refer to results of this type as {\it the theta correspondence for
  dummies}: replacing $\Hom$ with $\EP$ makes this more elementary
statement true.  See \cite{dummies}.
Hopefully such easy results for $\EP$,
together with vanishing results for higher $\Ext$
groups, will give results about $\Theta(\pi)$ and $\theta(\pi)$.
See Question \ref{q:question} and Example \ref{ex:gln}.
Moreover  non-trivial higher $\Ext$ groups may help to clarify the structure of $\Theta(\pi)$. 

Similar results hold for type I dual pairs. In order to keep the
notation simple we consider only the case of orthogonal-symplectic
dual pairs, see Section \ref{s:typeI}.
 It is clear that the proofs go
through for general type I dual pairs, using \cite{mvw}. 

It would be interesting to consider the case of real groups, say in
the context of $({\mathfrak g},K)$-modules, where we expect results
similar to what we obtain here for $p$-adic groups.

\sec{Some background}
\label{s:background}

In this section we introduce notation and prove some basic results needed later on. For background on representations of $p$-adic groups,  see \cite{casselman_padic} and \cite{rumelhart_bernstein}.

Suppose $G$ is a $p$-adic group. Let $\Cat=\Cat_G$ be the category of smooth representations of $G$, 
and let $\Ob$ denote the objects of this category. For $X$ any $G$-module, let $X^\infty\in\Ob$ be the 
submodule of smooth vectors.
It is a union of $X^K$, the space of $K$-fixed vectors, as $K$ runs over all open compact subgroups of $G$. 
In some cases $X$ is a module for two different groups $G$ and $H$, in which case we will 
be specific about denoting $X^\infty$ as $X^{G-\infty}$ or $X^{H-\infty}$. 
We work entirely in the setting of smooth representations.

For $Y\in\Ob$ let  $Y^*=\Hom_\C(Y,\C)$ be the algebraic dual, and let 
$Y^\vee=(Y^*)^{\infty}\in\Ob$ be the smooth dual.
For $H$  a closed subgroup of $G$, let $\Res^G_H$ be the restriction functor from 
$\Cat_G$ to $\Cat_H$, let $\Ind_H^G$  be  (smooth) induction from $\Cat_H$ to $\Cat_G$, 
and let $\ind_H^G$ be compact induction.

\begin{lemma}
\label{l:frob1} Let $K$ be an open compact subgroup of $G$. 
Suppose $X\in\ObK$ and $Y\in\ObG$. Then 
\begin{enumerate}
\item[(1)] $\Hom_G(\ind_K^G(X),Y)\simeq \Hom_K(X,\Res^G_K(Y))$. 
\item[(2)] $\ind_K^G(X)$ is a projective $G$-module. 
\end{enumerate}
\end{lemma}
\begin{proof} The first statement is a standard version of  Frobenius reciprocity; it implies that $\Hom_G(\ind_K^G(X), - )$ is an exact functor, so
$\ind_K^G(X)$ is projective. 
\end{proof} 

\begin{lemma} 
\label{l:dual_proj} Let $X$ be a smooth module for $G\times H$. 
If $X^K$ is a projective $H$-module for every 
open compact subgroup $K$ of $G$ then $X$ is a projective $H$-module. 
\end{lemma} 
\begin{proof} 
Fix an open compact subgroup $K$ of $G$. 
For $\tau$,  a smooth  irreducible  representation of $K$, 
let $X_\tau$ be the $\tau$-isotypic subspace of $X$. Let  $K_\tau$ be the kernel of $\tau$. 
Then $X_\tau$ is a direct summand of $X^{K_{\tau}}$ and is, therefore, projective. 
Furthermore $X=\oplus_{\tau} X_\tau$ is a direct sum of projective modules, hence projective.
\end{proof} 

\begin{lemma}
\label{l:ind_proj}
 Let $Q$ be a closed subgroup of  $G$ such that $Q\backslash G$ is compact.  Let
$X$ be a smooth representation of $Q \times H$, projective as $H$-module. Then $\Ind_Q^G X$, with the natural action of $H$, is 
a projective  $H$-module. 
\end{lemma}
\begin{proof} By Lemma \ref{l:dual_proj}, it suffices to prove that $(\Ind_Q^G X)^K$ is $H$-projective for every open compact subgroup $K$ of $G$. Write 
$G=\cup_i Q g_i K$ for a finite set of elements $g_i$ in $G$.  Let $K_i=g_i K g_i^{-1} \cap Q$. Then 
\[ 
(\Ind_Q^G X)^K= \oplus_i X^{K_i} 
\] 
where the isomorphism is given by evaluating $f\in \Ind_Q^G X$ at the points $g_i$. 
The Lemma follows  since $X^{K_i}$ are summands of $X$ and hence $H$-projective. 
\end{proof} 

Let $\S(G)$ be the Schwartz space of the locally constant compactly supported functions on $G$.
This is a module for $G\times G$ by the left and right translation actions.

\begin{lemma}
\label{l:S(G)}
\begin{enumerate}
\item[(1)] $\S(G)$ is a projective module for the right action of $G$.
\item[(2)] For any smooth (left) $G$-module $X$, $\Hom_G(\S(G),X)^{\infty}\simeq X$  as (left) $G$-modules, where 
$\Hom_G(\S(G),X)$ is
defined to be the space of homomorphisms $\lambda: \S(G) \rightarrow X$ with $\lambda(R_g f) = g\lambda(f)$ for all $f \in \S(G)$. 
\end{enumerate}
\end{lemma}
\begin{proof} Projectivity of $\S(G)$ is usually attributed to P. Blanc \cite{blanc}. We give an independent and rather simple proof. 
Let $K$  be an open compact subgroup of $G$, acting from the left. 
Then $\S(G)^K = \ind_{K}^G(\C)$ and this is projective by Lemma \ref{l:frob1}. 
 Hence $\S(G)$ is a projective $G$-module by Lemma \ref{l:dual_proj}. 

We now prove (2). For any compact open subgroup $K$ of $G$ we have
$$
\begin{aligned}
\Hom_{G}(\S(G),X)^K&\simeq \Hom_{G}(\S(K\backslash G),X)\\
&\simeq\Hom_G(\ind_K^G(\C), X)\\
&\simeq\Hom_K(\C, X) \\
&\simeq X^K.
\end{aligned}
$$
The isomorphism  $\Hom_{G}(\S(G),X)^K\simeq X^K$  is 
given  by $\varphi \mapsto \frac{1}{vol(K)}\varphi(1_K)$, where $1_K$ is the characteristic function of $K$.

These isomorphisms as $K$ varies are compatible, and therefore give an isomorphism  
$T: \Hom_{G}(\S(G),X)^\infty \simeq X$, defined by $T(\varphi) = \frac{1}{vol(K)} \varphi(1_K)$ for
$\varphi \in \Hom_{G}(\S(G), X)^{\infty}$, where $K$ is any compact open subgroup of $G$ which leaves $\varphi$ invariant. 
\end{proof}

Now assume that $G$ is reductive, $P=MN$ is a parabolic subgroup of $G$, 
and $\delta_P$ is  the modulus character of $P$:
$\delta_P(mn)=|\det(\Ad_{\frn}(m))|$ \cite[3.1]{casselman_padic}
For $X\in\ObM$ we write $i_P^G(X)$ for normalized induction ($X$ is pulled back to $P$):
$i_P^G(X)=\Ind_P^G(\delta_P^{\frac12}X)$.
Then $i_P^G$ preserves unitarity, and 
\begin{equation}
\label{e:inddual}
i_P^G(X)^\vee=i_P^G(X^\vee).
\end{equation}
For $X\in \ObG$ write $r^G_P(X)\in\ObM$ for the normalized Jacquet module of $X$:
$r^G_P(X)=H_0(N,X)\delta_P^{-\frac12}$.

\begin{lemma} \label{L:dual_pair_corr} 
Let $P=MN$ be a parabolic subgroup of  $G$. Let $V$ be an $M\times H$-module. Then, for every smooth $H$-module $U$, with trivial action of $G$, 
 we have the following natural isomorphism of smooth $G$-modules 
\[ 
 \Hom_{H} (i_{P}^{G} (V), U)^{\infty} \cong i_{P}^{G} (\Hom_{H}(V,U)^{M-\infty}). 
 \]  
\end{lemma} 
\begin{proof} 
It suffices to show that we have an isomorphism of $G\times H$-modules 
\[ 
 \Hom_{\mathbb C} (i_{P}^{G} (V), U)^{\infty} \cong  i_{P}^{G} (\Hom_{\mathbb C}(V,U)^{M-\infty})
 \] 
 and then the proposition follows by taking $H$-fixed vectors on both sides. Note that, if $U=\mathbb C$,  this is the well known statement 
 $i_P^G (V)^{\vee}\cong i_P^G(V^{\vee})$. The proof is the same.  More precisely,  $f\in i_{P}^{G} (\Hom_{\mathbb C}(V,U)^{M-\infty})$ 
 defines $\ell_f \in  \Hom_{\mathbb C} (i_{P}^{G} (V), U)^{\infty}$ by 
 \[ 
 \ell_f(f')= \int_{P\backslash G} f(g)(f'(g)) ~dg
 \] 
 for every $f'\in i_{P}^{G} (V)$.  One checks that $f\mapsto \ell_f$ is an isomorphism by doing so at the level of $K$-fixed vectors. The map is 
 determined by fixing a measure on $G$. 
\end{proof}

\sec{Euler-Poincare characteristic}

For  background on statements in this section see \cite{prasadext}. Assume that $G$ is a reductive group.

\begin{lemma}
\label{l:basic}
The category  $\Cat$ of smooth representations of $G$ has enough projectives and enough injectives.
Therefore, for $X,Y\in\Ob$,  we can define the complex vector spaces $\Ext_G^i(X,Y)$ for all $i\ge 0$ with the following
properties: 

\begin{enumerate}
\item $\Ext^0_G(X,Y)\simeq\Hom_G(X,Y)$;
\item $\Ext_G^i(X,Y)=0$ if $Y$ is injective. In general, 
$\Ext_G^i(X,Y)$ can be computed using an injective resolution of $Y$;
\item $\Ext_G^i(X,Y)=0$ if $X$ is projective. In general, 
$\Ext_G^i(X,Y)$ can be computed using a projective resolution of $X$.
\end{enumerate}
\end{lemma}

Now suppose $P=MN$ is a parabolic subgroup of $G$.
Let $\overline P=M\overline N$ be the opposite parabolic.
We have two versions of Frobenius reciprocity. 
\begin{lemma} For $X\in\ObG$  and  $Y\in\ObM$ 
\hfil
\label{l:frob2} 
\begin{enumerate}
\item $\Ext_G^i(X,i_P^G(Y))\simeq \Ext^i_M(r^G_P(X),Y)$
\item $\Ext_G^i(i_P^G(X),Y)\simeq \Ext^i_M(X,r_{\overline P}^G(Y))$
\end{enumerate}
\end{lemma}

We also need a version of the Kunneth formula \cite{prasadext}.

\begin{lemma}
\label{l:kunneth}
Suppose $G_1,G_2$ are  reductive $p$-adic groups, 
and $X_i,Y_i$ are smooth representations of $G_i$. 
Furthermore assume that $X_1$ is admissible. Then
$$
\Ext^i_{G_1\times G_2}(X_1\bx X_2,Y_1\bx Y_2)\simeq\bigoplus_{j+k=i}
\Ext^j_{G_1}(X_1,Y_1)\otimes \Ext^k_{G_2}(X_2,Y_2)
$$
\end{lemma}

See \cite{prasadext}. 
For $X,Y\in\ObG$ {\it assume} that $\Ext_G^i(X,Y)$ are finite dimensional for all $i\ge 0$, and 0 for $i$ large enough. 
Then the Euler-Poincare characteristic is defined to be
\begin{equation}
\EP_G(X,Y)=\sum_i(-1)^i\Ext_G^i(X,Y).
\end{equation}
This is a well-defined 
element of the Grothendieck group of finite dimensional vector spaces. 
If 
$$
0\rightarrow X_1\rightarrow X_2\rightarrow\dots\rightarrow X_n\rightarrow 0
$$
is an exact sequence  of smooth $G$-modules, and $Y\in\ObG$,  then
$$
\sum_j(-1)^j\EP_G(X_j,Y)=\sum_j(-1)^j\EP_G(Y,X_j)=0.
$$
If 
$$
0=X_0\subset X_1\subset\dots\subset X_n=X
$$
is a filtration by smooth $G$-submodules, with successive quotients $W_i=X_i/X_{i-1}$, then
\begin{equation}
\EP(X,Y)=\sum_{i=1}^{n}\EP(W_i,Y),
\quad
\EP(Y,X)=\sum_{i=1}^{n}\EP(Y,W_i).
\end{equation}

\begin{proposition}
\label{p:ep=0}
If $G$ is a reductive group and 
$X,Y$ are smooth $G$-modules of  finite length,  then $\Ext_G^i(X,Y)$ are finite dimensional 
for all $i$ and 0 for $i$ greater than the split rank of $G$. Moreover, if $G$ has non-compact center, then $\EP_G(X,Y)=0$.  
\end{proposition}
\begin{proof} The finite dimensionality of $\Ext^i(X,Y)$ and vanishing beyond the split rank are well-known general facts. We only prove the 
vanishing of Euler-Poincare characteristic for groups having non-compact center.
It suffices to prove this statement for a normal subgroup $G_0$ of finite index in $G$ which we assume has the form 
 $G_0=G_1 \times G_2$ where 
$G_1 \supseteq G^{\rm der}$ and $G_2 \subseteq Z(G)$ with $G_2 \cong \Z$. (These groups are not necessarily algebraic.) 
Decomposing $X$ and $Y$ as direct sums of irreducible representations for $G_0$,  it suffices to assume that $X$ and $Y$ themselves are irreducible as $G_0$-modules. 
Write $X= U_1 \otimes U_2$, and $Y=V_1 \otimes V_2$
where $U_1, V_1$ are irreducible modules for $G_1$ which are just the restrictions of the smooth modules $X, Y$ of $G_0$ to $G_1$; $U_2, V_2$ are the one 
dimensional representations on which $G_2$ operates by the central characters for the action of $G_0$ on $X$ and $Y$ respectively restricted to $G_2\simeq\Z \subset Z(G_0)$.
It suffices to prove that 
\[ 
\EP_{G_1\times G_2}(U_1\boxtimes U_2, V_1\boxtimes V_2)=0. 
\] 
By the K\"unneth formula, 
\[ 
\EP_{G_1\times G_2}(U_1\boxtimes U_2, V_1\boxtimes V_2)=\EP_{G_1}(U_1, V_1) \otimes \EP_{G_2}(U_2, V_2)  
\] 
Since $U_2,V_2$ are one dimensional representations of $G_2=\Z$, $\EP_{G_2}(U_2, V_2) =\EP_{\C[\Z]}(U_2, V_2) =0$, and the proposition follows. 
\end{proof}

Suppose that $H$ is another $p$-adic reductive group, $X\in Ob(\Cat_{G\times H})$ and $Y\in Ob(\Cat_H)$. 
Then $G$ acts on 
$\Ext_H^i(X,Y)$ via its action on $X$. This module is not necessarily smooth,
so we take its smooth vectors for the action of $G$:
$$
\Ext_H^i(X,Y)^{\infty}\in \ObG.
$$

We would like to use
$\Ext_H^i(X,Y)^{\infty}$,  to construct  the Euler-Poincare characteristic 
as an element of the Grothendieck group of finite length representations. For this we need to know that 
$\Ext_H^i(X,Y)^{\infty}$ is a  finite length smooth $G$-module.

\begin{definition}
Fix $X\in Ob(\mathcal C_{G\times H})$ and $Y\in\ObH$. 
Assume
\begin{equation}
\Ext_H^i(X,Y)^{\infty}\quad\text{is a finite length smooth $G$-module for all i}.
\end{equation}
Define $\EP_H(X,Y)^{\infty}=\sum_i (-1)^i\Ext_H^i(X,Y)^\infty$. This is a well defined element in the Grothendieck group 
of finite length representations.
\end{definition}
In practice we will always assume $Y$ (but not $X$) has finite length as an $H$-module.  

\sec{The Theta Correspondence}
\newcommand{\gs}{^{\infty}}

Consider a dual pair of subgroups $(G,H)$ of the metaplectic group
$\Mp(2n)$.  Fix an additive character of our $p$-adic field $F$ and let $\omega$ be the
corresponding oscillator representation of $\Mp(2n)$.

Consider the theta correspondence for the dual pair $(G,H)$. See
\cite{mvw}.  Suppose $\pi$ is an irreducible  representation of $G$.  Let 
$$
\omega(\pi)=\bigcap_{f\in\Hom_G(\omega,\pi)}\Ker(f).
$$
This is a $G\times H$-submodule of $\omega$. Set
$$
\omega[\pi]=\omega/\omega(\pi).
$$
This $G\times H$-module is  the maximal $\pi$-isotypic quotient of $\omega$. 
By \cite{mvw}, 
there is a smooth $H$-module $\Theta(\pi)$, unique up to equivalence, such that 
$\omega[\pi]\simeq\pi\boxtimes\Theta(\pi)$. 
By the Howe conjecture, now proved in  generality by Gan-Takeda \cite{gan_takeda_howe_conjecture}, $\Theta(\pi)$  
has a  unique irreducible quotient, which is denoted by  $\theta(\pi)$.

\begin{proposition}
\label{p:bigtheta}
$$
\Theta(\pi)^*\simeq \Hom_G(\omega,\pi)\text{ and } 
\Theta(\pi)\simeq \Hom_G(\omega,\pi)^\vee.
$$
\end{proposition}

\begin{proof}
By the definition of the maximal isotypic quotient,
$$
\begin{aligned}
\Hom_G(\omega,\pi)&\simeq \Hom_G(\pi\boxtimes\Theta(\pi),\pi)\\
&\simeq \Hom_G(\pi,\pi) \boxtimes
\Hom_\C(\Theta(\pi),\C) \\
&\simeq \Theta(\pi)^*
\end{aligned}
$$
This proves the first assertion. 
Taking the smooth vectors on both sides gives 
$
\Hom_G(\omega,\pi)^{\infty}\simeq \Theta(\pi)^\vee.
$
By \cite[III.5]{mvw} $\Theta(\pi)$ is a finite length $H$-module, so admissible,
so by \cite[Proposition 7]{rumelhart_bernstein} $\Theta(\pi)^{\vee^\vee}\simeq\Theta(\pi)$.
Take the smooth dual of both sides to conclude 
$\Theta(\pi)\simeq \Hom_G(\omega,\pi)^\vee$.
\end{proof}

\sec{Type II dual pairs}

Consider the oscillator representation $\omega$  of the dual pair $(\GL(m),\GL(n))$.
Let $\omega_0$ be the geometric representation of this dual pair on 
$\S(M_{m\times n})$:
\begin{equation}
\label{e:action}
\omega_0(g,h)(f)(x)=f(g\inv xh)\quad((g,h)\in\GL(m)\times\GL(n)).
\end{equation}
We use the standard normalization of the oscillator representation: set 
\begin{equation}
\label{e:xi}
\xi(g,h)=|\det(g)|^{-n/2}|\det(h)|^{m/2}
\end{equation}
and  define the oscillator representation to be:
\begin{equation}
\omega=\omega_0\otimes\xi.
\end{equation}
This is a unitary representation of $\GL(m)\times \GL(n)$ on which $\GL(1)$ embedded
as scalar matrices $(\lambda I_m,\lambda I_n) \in \GL(m) \times \GL(n)$  acts trivially. 

Let $t=\min(m,n)$. Consider the filtration
\begin{equation}
\label{e:filtration}
0=\omega_{t+1}\subset \omega_{t}\subset\dots\subset\omega_{0}=\omega
\end{equation}
where $\omega_k$ is the Schwartz space of functions supported on matrices in $M_{m\times n}(F) $ 
of rank $\ge k$. 
Let $\Omega_k$ be the set of matrices in $M_{m\times n}(F) $ of rank $k$.
Then $\omega_k/\omega_{k+1}\simeq \S(\Omega_k)$ $(0\le k\le t)$. 

For $\pi$ an irreducible representation of $\GL(m)$, we will show that  $\EP_{\GL(m)}(\S(\Omega_k),\pi)^{\infty}$ is well defined, and at the 
same time compute it; then we will calculate
 $\EP_{\GL(m)}(\omega,\pi)^{\infty}$ as the direct sum of these.

\subsec{$\EP_{\GL(m)}(\S(\Omega_k),\pi)^{\infty}$}

First of all by \cite{mvw}:
$$
\S(\Omega_k)\simeq\xi\Ind_{P_k\times Q_k}^{\GL(m)\times \GL(n)}(\S(\GL(k))\boxtimes 1).
$$
Note that the induction is unnormalized, and we've included the character $\xi$.
Here the parabolic subgroups and Levi factors are:
$$
M_k=\GL(k)\times \GL(m-k)\subset P_k=
\begin{pmatrix}
  *&*\\0&*  
\end{pmatrix}
\subset \GL(m)
$$
and
$$
L_k=\GL(k)\times \GL(n-k)\subset Q_k=
\begin{pmatrix}
  *&0\\ *&*
\end{pmatrix}
\subset \GL(n).
$$
We rearrange terms:
$$
M_k\times L_k\simeq [\GL(k)\times \GL(k)]\times [\GL(m-k)\times \GL(n-k)],
$$
and $\S(\GL(k))\boxtimes 1$ is a representation of 
$M_k\times L_k$ with respect to this decomposition.

So
$$
\Ext^i_{\GL(m)}(\S(\Omega_k),\pi)=
\Ext^i_{\GL(m)}(\xi\Ind_{P_k\times Q_k}^{\GL(m)\times \GL(n)}(\S(\GL(k))\boxtimes 1),\pi).
$$
To compute $\Ext^i_{\GL(m)}$, we only need the action of $\GL(m)$. So write
$$
\Ind_{P_k\times Q_k}^{\GL(m)\times \GL(n)}(\S(\GL(k))\boxtimes 1)
=
\Ind_{P_k}^{\GL(m)}\big\{\Ind_{Q_k}^{\GL(n)}(\S(\GL(k))\boxtimes 1)\boxtimes 1\big\}.
$$

Now apply Frobenius reciprocity, Lemma \ref{l:frob2}(2). Write $\nu_k(j)$ for the character $|\det|^j$ of $\GL(k)$.
$$
\begin{aligned}
&\Ext^i_{\GL(m)}(\xi\Ind_{P_k}^{\GL(m)}\big\{\Ind_{Q_k}^{\GL(n)}(\S(\GL(k))\boxtimes 1) \boxtimes 1 \big \},\pi)\\
&=\Ext^i_{\GL(m)}(\nu_m(-\tfrac n2)i_{P_k}^{\GL(m)}\big\{\delta_{P_k}^{-\tfrac12}\Ind_{Q_k}^{\GL(n)}(\S(\GL(k))\boxtimes 1) \boxtimes 1 \big\},\pi)\otimes 
\nu_n(-\tfrac m2)\\
&=\Ext^i_{\GL(k)\times \GL(m-k)}(\Ind_{Q_k}^{\GL(n)}(\S(\GL(k))\boxtimes 1)\boxtimes 1,\nu_m(\tfrac n2)\delta_{P_k}^{-\tfrac12}r^{\GL(m)}_{\overline P_k}(\pi))\otimes\nu_n(-\tfrac m2)
\end{aligned}
$$
Here $\GL(k)$ is acting on $\Ind_{Q_k}^{\GL(n)}(*)$ by its action on
$\S(\GL(k))$, and $\GL(m-k)$ is acting trivially (the second $\bx$).

Write
\begin{equation}
\label{e:jacquet}
\nu_m(\tfrac n2)\delta_{P_k}^{-\tfrac12}r^{\GL(m)}_{\overline P_k}(\pi)=\sum_{j=1}^\ell\sigma_j\boxtimes \tau_j
\end{equation}
with $\sigma_j\boxtimes\tau_j$ an irreducible representation of $\GL(k)\times \GL(m-k)$.

Now we  compute
$$
\Ext^i_{\GL(k)\times \GL(m-k)}(\Ind_{Q_k}^{\GL(n)}(\S(\GL(k))\boxtimes 1)\boxtimes 1,\sigma_j\boxtimes\tau_j),
$$
where the first $1$ denotes the trivial representation of $\GL(n-k)$, and the second $1$ denotes 
the trivial representation of $\GL(m-k)$. By the Kunneth formula (Lemma \ref{l:kunneth}), this is equal to
\begin{equation}
\label{e:kunneth}
\sum_{p=0}^i
\Ext^p_{\GL(k)}(\Ind_{Q_k}^{\GL(n)}(\S(\GL(k))\boxtimes 1),\sigma_j)\otimes
\Ext^{i-p}_{\GL(m-k)}(1,\tau_j)
\end{equation}
By Lemma \ref{l:ind_proj} the induced representation 
$\Ind_{Q_k}^{\GL(n)}(\S(\GL(k))\boxtimes 1)$ 
is projective as a representation of $\GL(k)$.
Therefore all terms in  \eqref{e:kunneth} with $p>0$ are $0$, and (summing over $j$ again) we get
\begin{equation}
\label{e:hom}
\begin{aligned}
&\Ext_{\GL(m)}^i(S(\Omega_k),\pi)\simeq\\
&\sum_{j=1}^\ell\Hom_{\GL(k)}(\Ind_{Q_k}^{\GL(n)}(\S(\GL(k))\boxtimes 1),\sigma_j)\otimes
\Ext^{i}_{\GL(m-k)}(1,\tau_j)\otimes\nu_n(-\tfrac m2)
\end{aligned}
\end{equation}
Now take the $\GL(n)$-smooth vectors on both sides.
We want to apply Lemma \ref{L:dual_pair_corr}  to the first factor on the right hand side,
so first 
we replace $\Ind_{Q_k}^{\GL(n)}$ with normalized induction.
\begin{equation}
\label{e:homnorm}
\Ind_{Q_k}^{\GL(n)}(\S(\GL(k))\boxtimes 1)\simeq 
i_{Q_k}^{\GL(n)}(\S(\GL(k))\nu_k(\tfrac{n-k}2)\boxtimes \nu_{n-k}(-\tfrac k2))
\end{equation}
and the first term on the right hand side of \eqref{e:hom} is
$$
\begin{aligned}
&\Hom_{GL(k)}(i_{Q_k}^{\GL(n)}(\S(\GL(k))\nu_k(\tfrac{n-k}2)\boxtimes \nu_{n-k}(-\tfrac k2)),\sigma_j)^{\infty}\\
&\simeq
i_{Q_k}^{\GL(n)}(\Hom_{\GL(k)}(\S(\GL(k))\nu_k(\tfrac{n-k}2),\sigma_j)^{\infty}\boxtimes\nu_{n-k}(\tfrac k2))\\
&\simeq 
i_{Q_k}^{\GL(n)}(\sigma_j\nu_k(\tfrac{-n+k}2)\boxtimes \nu_{n-k}(\tfrac k2))
\end{aligned}
$$
where the final isomorphism is by Lemma \ref{l:S(G)}(2).

This proves the following intermediate result.
Recall that  $\sigma_j\bx\tau_j$, 
an irreducible representation of $\GL(k)\times \GL(m-k)$, is given in \eqref{e:jacquet}.
\begin{proposition} \label{p:finitelength} For an irreducible smooth representation $\pi$ of $\GL(m)$, we have an isomorphism of 
$\GL(n)$-modules:
\label{e:finitelength}
$$\Ext_{\GL(m)}^i(\S(\Omega_k),\pi)^{\infty}\simeq
\sum_{j=1}^{\ell} i_{Q_k}^{\GL(n)}(\sigma_j\nu_{k}(\tfrac{-n+k-m}2)\boxtimes \nu_{n-k}(\tfrac{k-m}2))\otimes\Ext^{i}_{\GL(m-k)}(1,\tau_j),
$$
with $\GL(n)$ acting trivially on the last factor.
In particular, $\Ext_{\GL(m)}^i(\S(\Omega_k),\pi)^{\infty}$  is a finite length $\GL(n)$-module.
\end{proposition}
Note that the right hand side is $0$ if $i>m-k$. 

An important special case is $m=k,i=0$. The Levi factor of $Q_m$ is $\GL(m)\times \GL(n-m)$.
Also $\ell=1$, and \eqref{e:jacquet} is simply $\nu_m(\frac n2)\pi=\sigma_1$.
Plugging this in gives
\begin{equation}
\label{e:specialcase}
\begin{aligned}
\Hom_{\GL(m)}(\S(\Omega_m),\pi)^{\infty}&\simeq
i_{Q_k}^{\GL(n)}(\pi\boxtimes 1)\\
\Ext^i_{\GL(m)}(\S(\Omega_m),\pi)^{\infty}&=0\quad (i>0).
\end{aligned}
\end{equation}

Now we can conclude  that the Euler-Poincare characteristic is well defined, 
and  \eqref{e:finitelength} yields:
$$
\EP_{\GL(m)}(\S(\Omega_k),\pi)^{\infty}\simeq
\sum_{j=1}^\ell i_{Q_k}^{\GL(n)}(\sigma_j\nu_{k}(\tfrac{-n+k-m}2)\boxtimes \nu_{n-k}(\tfrac{k-m}2))\otimes\EP_{\GL(m-k)}(1,\tau_j).
$$
By Proposition \ref{p:ep=0},  $\EP_{\GL(m-k)}(1,\tau_j)=0$ unless $k=m$, and 
if $k=m$ 
 \eqref{e:specialcase} gives
$$
\EP_{\GL(m)}(\S(\Omega_m),\pi)^{\infty}=i_{Q_m}^{\GL(n)}(\pi\bx 1)
$$
This proves:
\begin{proposition}
\label{p:EPomegak}
Let $\Omega_k$ be the set of $m\times n$ matrices (over $F$) of rank $k\le \min(m,n)$.
Consider the action of $\GL(m)\times \GL(n)$ on $\S(\Omega_k)$ given by \eqref{e:action}. 
Then for an irreducible smooth representation $\pi$ of $\GL(m)$,
$$
\EP_{\GL(m)}(\S(\Omega_k),\pi)^{\infty}=
\begin{cases}
0&k<m\\
i_{Q_m}^{\GL(n)}(\pi\boxtimes 1)&k=m
\end{cases}
$$
\end{proposition}

Recall that $\omega$ has a filtration with successive quotients $\S(\Omega_k)$ ($0\le k \le \min(m,n)$). 

\begin{proposition} For an irreducible smooth representation $\pi$ of $\GL(m)$,
\label{p:EPomega}
$$
\EP_{\GL(m)}(\omega,\pi)^{\infty}=  
\sum_{k=0}^{\min(m,n)}
\EP_{\GL(m)}(\S(\Omega_k),\pi)^{\infty}
$$
\end{proposition}

This would be a trivial consequence of  Proposition \ref{p:EPomegak}, except for a technical issue related to taking smooth vectors,  
so we defer the proof to the next sub-section.
Together with Proposition
\ref{p:EPomegak},  this implies the main result of this section.

\begin{theorem}
\label{t:epgln}
Let $\pi$ be an irreducible representation of $\GL(m)$. 
Then
$$
\EP_{\GL(m)}(\omega,\pi)^{\infty}=
\begin{cases}
0&n<m\\ 
i_{Q_m}^{\GL(n)}(\pi\boxtimes1)&n\ge m
\end{cases}
$$
where $Q_m$ is a parabolic subgroup of $\GL(n)$ with  Levi subgroup $\GL(m)\times \GL(n-m)$.
\end{theorem}

\subsec{$\EP_{\GL(m)}(\omega,\pi)^{\infty}$}

We prove Proposition \ref{p:EPomega}.

\begin{lemma}
\label{l:longexact}
Suppose $0\rightarrow A\rightarrow B\rightarrow C\rightarrow 0$ is a short exact sequence of smooth $G\times H$-modules, 
and $X$ is a smooth $H$-module. 
Then the sequence of smooth $G$-modules
$$
\begin{aligned}
&0\rightarrow\Hom_H(C,X)^{\infty}\rightarrow
\Hom_H(B,X)^{\infty}\rightarrow
\Hom_H(A,X)^{\infty}\rightarrow\dots\\
&\dots\rightarrow\Ext^{i}_H(C,X)^{\infty}\rightarrow
\Ext^{i}_H(B,X)^{\infty}\rightarrow
\Ext^{i}_H(A,X)^{\infty}\rightarrow\dots\\
\end{aligned}
$$
is exact.
\end{lemma}

\begin{proof}
The given sequence is exact before taking the smooth vectors.
We claim this remains true after taking the smooth vectors. 

This is not immediate since the functor $\pi \rightarrow \pi^\infty$
from the category of all $G$-modules to the category of smooth
$G$-modules is only left exact. Indeed, for $G$ a non-discrete
topological group, consider the right exact sequence of $G$-modules
\[ 
\mathbb C[G] \rightarrow \mathbb C \rightarrow 0 
\] 
where $\mathbb C[G]$ is the group algebra of $G$, i.e. the space of finite sums $\sum_{g\in G} c_g g$, where $c_g\in\mathbb C$, and the first 
arrow is $\sum_{g\in G} c_g g\mapsto \sum_{g\in G} c_g$. Since $G$ is non-discrete,  $\mathbb C[G]^{\infty}=0$, so the sequence is not exact after taking smooth vectors. 

The essence of the proof below is that for $G$-modules of the form $\Hom_H(Y,X)$  where  $X$ is a smooth module 
of $H$ and $Y$ of $G \times H$, taking smooth vectors for $G$ is an exact functor.

Let $K$ be an open compact subgroup of $G$.
Since taking $K$-fixed vectors is an exact functor on smooth representations of $G$ \cite{casselman_padic}, 
the sequence
\begin{equation}
\begin{aligned}
&0\rightarrow\Hom_H(C^K,X)\rightarrow
\Hom_H(B^K,X)\rightarrow
\Hom_H(A^K,X)\rightarrow\dots\\
&\dots\rightarrow\Ext^{i}_H(C^K,X)\rightarrow
\Ext^{i}_H(B^K,X)\rightarrow
\Ext^{i}_H(A^K,X)\rightarrow\dots\\
\end{aligned}
\end{equation}
is exact.

Let $Y$ be a smooth $G\times H$-module. Taking an injective resolution of $X$:
$$0 \rightarrow   X \rightarrow X_0  \rightarrow  X_1 \rightarrow \cdots,$$
we can calculate $\Ext^i_{H}(Y,X)$ by using the chain complex $\Hom_H(Y,X_i)$.  
Write $Y=\oplus_{\tau} Y_{\tau}$ as a direct sum over irreducible  representations of $K$, 
 where $K$ acts on $Y_{\tau}$ as a multiple of $\tau$.
 Then $\Hom_H(Y,X_i)$  is a direct product of spaces 
$\Hom_H(Y_{\tau},X_i)$ over all $\tau$,  for every $i$.  Let $\Ext^i_{H}(Y_{\tau},X)$ be the cohomology of the complex $\Hom_H(Y_{\tau},X_i)$. Then
$\Ext_H^i(Y,X)$  is a direct product of spaces $\Ext_H^i(Y_{\tau},X)$ over all $\tau$, for every $i$. 

For every $\tau$,  $K$ acts on the complex $\Hom_H(Y_{\tau},X_i)$ through a finite quotient, hence 
$\Ext^i_{H}(Y_{\tau},X)^K$ is calculated using the complex $\Hom_H(Y_{\tau},X_i)^K=\Hom_H(Y_{\tau}^K, X_i)$. Note that $Y_{\tau}^K=0$ unless 
$\tau$ is the trivial representation of $K$ and then $Y_{\tau}^K=Y^K$.  
Since $\Ext_H^i(Y,X)$  is a direct product of spaces $\Ext_H^i(Y_{\tau},X)$, it follows that 
$\Ext_H^i(Y,X)^K\simeq \Ext_H^i(Y^K,X)$ for all $i$. This proves that the following is an exact sequence for all $K$:
\begin{equation}
\begin{aligned}
&0\rightarrow\Hom_H(C,X)^K\rightarrow
\Hom_H(B,X)^K\rightarrow
\Hom_H(A,X)^K\rightarrow\dots\\
&\dots\rightarrow\Ext^{i}_H(C,X)^K\rightarrow
\Ext^{i}_H(B,X)^K\rightarrow
\Ext^{i}_H(A,X)^K\rightarrow\dots\\
\end{aligned}
\end{equation}
The lemma now follows since the subspace of $G$-smooth vectors is a union of subspaces of $K$-fixed vectors for all $K$. 
\end{proof}

\begin{proposition}
\label{p:epomega}
Suppose $\pi$ is an irreducible representation of $G$. Then for $0\le k\le\min(m,n)$, 
$\Ext^i_G(\omega_k,\pi)^\infty$ is a $G'$-module of finite length, 
and $\EP(\omega_k,\pi)^\infty$ is a well-defined element of the Grothendieck group of finite length representations of $G'$.
\end{proposition}

\begin{proof}
Return to the filtration \eqref{e:filtration}, and for $0\le k\le t=\min(m,n)$ consider the exact sequence
\begin{equation}
\label{e:exact}
0\rightarrow \omega_{k+1}\rightarrow \omega_k\rightarrow \S(\Omega_k)\rightarrow 0
\end{equation}
Apply Lemma \ref{l:longexact} to this, with $X=\pi$. 
Part of the resulting long exact sequence is:
$$
\begin{aligned}
\Ext^i_{\GL(m)}(\S(\Omega_k),\pi)^\infty
&\rightarrow
\Ext^i_{\GL(m)}(\omega_k,\pi)^\infty\rightarrow\\
&\rightarrow\Ext^i_{\GL(m)}(\omega_{k+1},\pi)^\infty
\rightarrow
\Ext^{i+1}_{\GL(m)}(\S(\Omega_k),\pi)^\infty
\end{aligned}
$$
By Proposition \ref{e:finitelength} the first and last terms have finite length. 
Since $\omega_{t+1}=0$  we conclude $\Ext^i_{GL(m)}(\omega_t,\pi)^\infty$ has finite length. 
By downward induction on $k$ the same holds for 
$\Ext^i_{GL(m)}(\omega_k,\pi)^\infty$ for all $k$. 
The assertion about $EP$ follows from this and  Lemma \ref{l:basic}(2).
\end{proof}

\begin{proof}[Proof of Proposition \ref{p:EPomega}]
By the preceding Lemma we conclude
$$
\EP(\omega_k,\pi)^{\infty}=
\EP(\S(\Omega_k),\pi)^{\infty}+
\EP(\omega_{k+1},\pi)^{\infty}.
$$
The Proposition follows by repeated applications of this, starting with $k=0$ and going through $k=\min(m,n)$. 
\end{proof}

\sec{Example: $(\GL(1),\GL(n))$}

Let $(G,H)=(\GL(1),\GL(n))$, acting on $\omega=\S(V)$ where $V=F^n$. 
The filtration is
$0\subset \S(\Omega_1)\subset \S(V)=\omega$, where $\Omega_1=V-\{0\}$.

\begin{lemma}
\label{l:omegaprojective}
The representation $\omega = \S(V)$ of $\GL(1)$ is a projective $\GL(1)$-module. 
\end{lemma} 
\begin{proof} Let $K$ and $K'$ be maximal compact subgroups of $\GL(1)$ and $\GL(n)$, respectively. Let
\[ 
\omega=\oplus_{\tau} \omega_{\tau} 
\] 
be the decomposition of $\omega$ into $K'$-isotypic  components. It suffices to prove that each summand is projective. If $\tau$ is not the 
trivial representation, then $ \S(\Omega_1)_{\tau}=\omega_{\tau}$. Since $ \S(\Omega_1)$ is projective, it follows that $\omega_{\tau}$ is 
projective. It remains to deal with the trivial $K'$-type. The group $K'$ stabilizes a lattice chain in $V$, and any smooth $K'$-invariant 
function is a linear combination of characteristic functions of the lattices in the chain. Since $\GL(1)$ acts transitively on the lattices in the chain 
with the stabilizer $K$, it follows that 
\[ 
\omega^{K'}\cong \ind_K^{\GL(1)} (1) 
\] 
and so is projective. \end{proof} 

Recall the definition of  the character $\xi=|\det|^{-\frac n2}\bx|\det|^{\frac12}$  of $G\times H$ \eqref{e:xi}.
We have the exact sequence
\begin{equation}
\label{e:S(V)}
0\rightarrow \S(\Omega_1)\rightarrow \omega \rightarrow\xi\rightarrow 0 
\end{equation}
which gives the exact sequence of 
$\GL(n)$-modules:
$$
\begin{aligned}
0
\rightarrow \Hom_{\GL(1)}(\xi,\chi)
\rightarrow  \Hom_{\GL(1)}(\omega,\chi) \\
\rightarrow  \Hom_{\GL(1)}(\S(\Omega_1),\chi)
\rightarrow \Ext^1_{\GL(1)}(\xi,\chi)
\rightarrow 0.  
\end{aligned}
$$
By Lemma \ref{l:longexact}, the sequence remains exact after taking the smooth vectors; 
take the smooth dual to give the exact sequence:
$$
\begin{aligned}
0
\rightarrow\Ext^1_{\GL(1)}(\xi,\chi)^\vee
\rightarrow  \Hom_{\GL(1)}(\S(\Omega_1),\chi)^\vee\\
\rightarrow\Hom_{\GL(1)}(\omega,\chi)^\vee
\rightarrow\Hom_{\GL(1)}(\xi,\chi)^\vee
\rightarrow 0.
\end{aligned}
$$
By \eqref{e:specialcase} and Proposition \ref{p:bigtheta},
$$
\begin{aligned}
\Hom_{\GL(1)}(\S(\Omega_1),\chi)^\vee&\simeq i_{Q_1}^{\GL(n)}(\chi^\vee\boxtimes 1)\\
\Hom_{\GL(1)}(\omega,\chi)^\vee&\simeq \Theta(\chi).
\end{aligned}
$$
By  Proposition \ref{p:ep=0},  $\EP_{\GL(1)}(\xi,\chi)=0$,  i.e.
$$
\Ext_{\GL(1)}^1(\xi,\chi)=\Hom_{\GL(1)}(\xi,\chi)=
\begin{cases}
  |\det|^{-\frac12}&\chi=|\det|^{-\frac n2}\\0&\text{else.}
\end{cases}
$$
If $\chi\ne|\det|^{-\frac n2}$ we conclude
$$
\Theta(\chi)\simeq i_{Q_1}^{\GL(n)}(\chi^\vee\boxtimes\C).
$$
On the other hand taking $\chi=|\det|^{-\frac n2}$ gives the exact sequence
$$
0
\rightarrow|\det|^{\frac12}
\rightarrow i_{Q_1}^{\GL(n)}(|\det|^{\frac n2}\boxtimes\C)
\rightarrow \Theta(|\det|^{-\frac n2})
\rightarrow|\det|^{\frac12}
\rightarrow0,
$$
which implies
$$
\theta(|\det|^{-\frac n2})=|\det|^{\frac12}.
$$
Note that  $\Theta(|\det|^{-\frac n2})$ and 
$i_{Q_1}^{\GL(n)}(|\det|^{\frac n2}\boxtimes\C)$ have the same image in the Grothendieck group.
However $|\det|^{\frac12}$ is a quotient of the former, and a submodule of the latter.

In particular if $n=1$ this proves
$$
\Theta(\chi)=\theta(\chi)=\chi^\vee\quad(\text{for all }\chi).
$$

If $n=2$ taking $\chi=|\det|^{-1}$ shows that 
$\Theta(|\det|^{-1})$ has 
$\text{Steinberg}|\det|^{\frac12}$ as a submodule, and 
$|\det|^{\frac12}\simeq \theta(|\det|^{-1})$ as a quotient (the opposite composition series
of the induced representation).

\begin{question}
\label{q:question}
Consider a dual pair $(G,G')$ in the metaplectic group  $\Mp(2n)$, such that
the split rank of $G$ is not greater than the split rank of $G'$.
Let $\omega$ be the oscillator representation  representation of $\Mp(2n)$. 
Is it true that $\omega$ is a projective module in the category of smooth representations of $G$?
In fact all we need is an affirmative answer to 
the (presumably weaker question): for every irreducible representation $\pi$ of $G$, is
\begin{equation}
\label{e:extvanish}
\Ext^i_G(\omega, \pi) = 0~~~~~~~~~\text{for all } i>0?
\end{equation}
A similar question is posed in 
\cite[Conjecture 2]{prasadext}. For $F$ an Archimedean field, similar questions may be posed 
in the category of $({\mathfrak g},K)$-modules.
\end{question}

Assuming \eqref{e:extvanish} then  $\Theta(\pi)^\vee=\EP_G(\omega,\pi)$; 
the right hand side is a more elementary object, and easier to compute. 
A similar discussion  applies when reducing the computation of  $\Hom$ to $\EP$ in 
branching problems, as discussed in \cite[Section 2]{prasadext}. 

\begin{example}
\label{ex:gln}
Consider the dual pair $(\GL(m),\GL(n))$ with $m\le n$, and suppose
$\pi$ is an irreducible representation of $\GL(m)$.  Assume 
$\Ext^i_{\GL(m)}(\omega,\pi)=0$ for $i>0$.  Then  by Theorem \ref{t:epgln}
\begin{equation}
\label{e:Thetaind}
\Theta(\pi)=i_Q^{\GL(n)}(\pi^\vee\boxtimes 1)
\end{equation}
(equality in the Grothendieck group of finite length $GL(n)$-modules). 
The computation of $\Theta(\pi)$ (as opposed to its irreducible quotient $\theta(\pi)$)
seems not to be available in the literature, even in the 
case of type II dual pairs.

If $\pi$ is unitary then $\theta(\pi)=\Theta(\pi)=i_Q^{\GL(n)}(\pi^\vee\boxtimes 1)$ 
since the induced representation is irreducible.
Also if $m=1$ then \eqref{e:extvanish} holds by Lemma \ref{l:omegaprojective}.
In all of these cases \eqref{e:Thetaind} holds.
\end{example}

\sec{Type I Dual Pairs}
\label{s:typeI}

Let $\O(N)$ be the isometry group of a nondegenerate  quadratic space of dimension $N$. 
Let $\omega$ be the oscillator representation for the dual pair
$(\Sp(2m),\O(N))$. We'll ignore the issue of covers, which play a non
essential role, while making the notation more cumbersome. 
If $N$ is even, the covers can be avoided altogether.

Suppose $P(t)=M(t)N(t)$ is the stabilizer of a $t$-dimensional isotropic
subspace of the symplectic space, and $\pi$ is an irreducible representation of $M(t)$. 
We want to compute
\begin{subequations}
\renewcommand{\theequation}{\theparentequation)(\alph{equation}}  
\label{e:typeI}
\begin{equation}
\Ext^i_{\Sp(2m)}(\omega,i_{P(t)}^{\Sp(2m)}(\pi))^{\infty}
\end{equation}
as well as $\EP$. 

By Frobenius reciprocity (Lemma \ref{l:frob2}(1)) 
\begin{equation}
\Ext^i_{\Sp(2m)}(\omega,i_{P(t)}^{\Sp(2m)}(\pi))
\simeq
\Ext^i_{M(t)}(r^{\Sp(2m)}_{P(t)}(\omega),\pi).
\end{equation}
Write
\begin{equation}
\label{e:M(t)}
M(t)=\GL(t)\times \Sp(2m-2t).
\end{equation}
For $0\le j\le \min(t,n)$,  let
\begin{equation}
P(t,j)=M(t,j)N(t,j)\subset M(t)
\end{equation}
be a parabolic subgroup of $M(t)$ 
where
\begin{equation}
M(t,j)=\GL(t-j)\times \GL(j)\times \Sp(2m-2t)
\end{equation}
and $N(t,j)\subset \GL(t)\subset M(t)$.
Let
\begin{equation}
Q(j)=L(j)U(j)\subset\O(N)
\end{equation}
be a parabolic subgroup of $\O(N)$ with Levi factor
$$
L(j)=
\GL(j)\times \O(N-2j).
$$
Let $n$ be the Witt index of $V$, so $V$ is the direct sum of an
anisotropic kernel $V_0$ and a hyperbolic space of dimension $2n$. We
will also consider the family of orthogonal spaces with the same
anistropic kernel $V_0$.  Such a space is determined by its Witt
index, so we write $\omega_{m',n'}$ for the oscillator representation
for the dual pair $(\Sp(2m'),\O(N'))$, where the orthogonal has dimension $N'=\dim(V_0)+2n$, anisotropic kernel $V_0$ and Witt index $n$.
With this convention $\omega=\omega_{m,n}$.

By \cite[3.IV.5]{mvw}, the representation $r^{\Sp(2m)}_{P(t)}(\omega_{m,n})$  of $M(t)\times \O(N)=\GL(t)\times\Sp(2m-2t)\times\O(N)$ 
has a filtration
\begin{equation}
\label{e:filtrationtypeI}
0=F_{t+1}\subset F_t\subset F_{t-1}\subset\dots\subset F_0=r^{\Sp(2m)}_{P(t)}(\omega_{m,n})
\end{equation}
 with subquotients
\begin{equation}
F_{j}/F_{j+1}\simeq i_{P(t,j)\times Q(j)}^{M(t)\times \O(N)}(\xi(t,j)\otimes\S(\GL(j))\bx\omega_{m-t,n-j})
\end{equation}
for some character $\xi(t,j)$ of $\GL(t-j)$.
The actions are:
$$
\begin{aligned}
&\GL(j)\subset M(t,j)\text{ acts on $\S(\GL(j))$}\text{ on the left},\\
&\GL(t-j)\subset M(t,j)\text{ acts by some character }\xi(t,j),\\
&\GL(j)\subset L(j)\text{ acts on $\S(\GL(j))$}\text{ on the right},\\
&(\Sp(2m-2t),\O(N-2j))\text{ acts by the oscillator representation $\omega_{m-t,n-j}$}.
\end{aligned}
$$
So we need to compute 
\begin{equation}
\Ext^i_{M(t)}(
i_{P(t,j)\times Q(j)}^{M(t)\times \O(N)}(\xi(t,j)\otimes\S(\GL(j))\bx\omega_{m-t,n-j}),\pi).
\end{equation}

By the second version of Frobenius reciprocity (Lemma \ref{l:frob2}(2)), applied to the induction 
from $P(t,j)$ to $M(t)$, this is isomorphic to
\begin{equation}
i_{Q(j)}^{\O(N)}(\Ext^i_{M(t,j)}
(\xi(t,j)\otimes\S(\GL(j))\bx\omega_{m-t,n-j}),r^{M(t)}_{\overline P(t,j)}(\pi)).
\end{equation}
Recall $M(t,j)=\GL(t-j)\times \GL(j)\times \Sp(2m-2t)$, and the tensor product appearing above is with respect to this decomposition.
Write
\begin{equation}
\pi=\pi_1\bx\pi_2
\end{equation}
with respect to the decomposition \eqref{e:M(t)}, and then
\begin{equation}
\label{e:sum}
r^{M(t)}_{\overline P(t,j)}(\pi)=\sum_{k=1}^\ell\sigma_{j,k}\bx\tau_{j,k}\bx\pi_2
\end{equation}
as a representation of $\GL(t-j)\times \GL(j)\times \Sp(2m-2t)$. 
So we need to compute
\begin{equation}
\label{e:extterm}
\Ext^i_{M(t,j)}
(\xi(t,j)\otimes\S(\GL(j))\bx\omega_{m-t,n-j}),\sigma_{j,k}\bx\tau_{j,k}\bx \pi_2)
\end{equation}
This is a representation of 
$L(j)=\GL(j)\times \O(N-2j)$. We need to know that the space of $L(j)$-smooth vectors has finite length.

By the Kunneth formula, this is a sum of (external tensor products of) terms of the following three types.
The first is
\begin{equation}
\Ext^a_{\GL(t-j)}(\xi(t,j),\sigma_{j,k})
\end{equation}
where $\xi(t,j)$ is a character and $\sigma_{j,k}$ is irreducible. 
By Proposition \ref{p:ep=0} this is finite dimensional (and trivial as a representation of $L(j)$).
The second type is
\begin{equation}
\Ext^b_{\GL(j)}(\S(\GL(j)),\tau_{j,k})
\end{equation}
where $\GL(j)$ is acting $\S(\GL(j))$ on the right, and $\tau_{j,k}$ irreducible. 
By  Lemma \ref{l:S(G)} this is $0$ if $b>0$, and for $b=0$ the space of  smooth vectors 
(with $\GL(j)$ acting on the left on $\S(\GL(j))$ and hence on $\Hom_{\GL(j)}(\S(\GL(j)),\tau_{j,k})$) 
is $\tau_{j,k}$. 
The  third term is
\begin{equation}
\Ext^c_{\Sp(2m-2t)}(\omega_{m-t,n-j},\pi_2)^{\infty},
\end{equation}
which if $\pi_2$ is supercuspidal is nonzero only if $c=0$, in which case it is of finite length for the corresponding orthogonal group.
\end{subequations}

We have proved the following intermediate result in the preceding paragraphs.

\begin{lemma}\label{lemma7.2}
Suppose $\pi = \pi_1 \boxtimes \pi_2$ is an irreducible representation of $M(t)=\GL(t)\times \Sp(2m-2t)$ with $\pi_2$ a supercuspidal representation of
$\Sp(2m-2t)$. 
Then $\Ext^i_{\Sp(2m)}(\omega_{m,n},i_{P(t)}^{Sp(2m)}(\pi))$ is a finite length $O(N)$-module for all $i$.
\end{lemma}

\begin{proposition}
Consider the oscillator representation  $\omega$ for the dual pair $(\Sp(2m),\O(N))$,
and suppose $\pi$ is an irreducible representation of $\Sp(2m)$. Then
$\Ext^i_{\Sp(2m)}(\omega,\pi)^{\infty}$ is a finite length $\O(N)$-module for all $i$.
\end{proposition}

\begin{proof} 
The proof of the Proposition is by  induction on $i$, by  an argument similar to \cite[Section 5, Lemma 3]{prasadext}. Thus we 
assume that we have proved the Proposition for all $i\le d$ and all $\pi$ irreducible, and therefore 
$\Ext^i_{\Sp(2m)}(\omega,\pi)^{\infty}$, $i \leq d$,  also has finite length as an $\O(N)$-module for $\pi$ of finite length as $\Sp(2m)$-module. 
The case $i=0$ is in \cite{kudla_induction} and \cite{mvw}. 

We now need the following lemma which is proved in the same manner as Lemma \ref{l:longexact}. 
\begin{lemma}
\label{l:longexact_2}
Suppose $0\rightarrow A\rightarrow B\rightarrow C\rightarrow 0$ is a short exact sequence of smooth $H$-modules, 
and $X$ is a smooth $G\times H$-module. 
Then the sequence
$$
\begin{aligned}
&0\rightarrow\Hom_H(X,A)^{\infty}\rightarrow
\Hom_H(X,B)^{\infty}\rightarrow
\Hom_H(X,C)^{\infty}\rightarrow\dots\\
&\dots\rightarrow\Ext^{i}_H(X,A)^{\infty}\rightarrow
\Ext^{i}_H(X,B)^{\infty}\rightarrow
\Ext^{i}_H(X,C)^{\infty}\rightarrow\dots\\
\end{aligned}
$$
is exact.
\end{lemma}

Let $\pi$ be any irreducible  representation of $\Sp(2m)$. Then $\pi$ sits in an exact sequence 
\[ 
0 \rightarrow \pi \rightarrow I \rightarrow J \rightarrow 0 
\] 
where $I$ is fully induced from a supercuspidal representation. Lemma \ref{l:longexact_2} now gives an exact sequence 
\[ 
\Ext^i_{\Sp(2m)}(\omega,J)^{\infty}\rightarrow 
\Ext^{i+1}_{\Sp(2m)}(\omega,\pi)^{\infty} \rightarrow \Ext^{i+1}_{\Sp(2m)}(\omega,I)^{\infty}.
\] 
The first term is of finite length for $i\le d$ by the inductive hypothesis, and
this holds by Lemma \ref{lemma7.2} for the last term since $I$ is fully induced from a
supercuspidal representation. This implies that 
the middle term has finite length for $i+1\le d+1$.
\end{proof}

We conclude that the space of $L(j)$-smooth vectors in
\eqref{e:extterm} has finite length, and $\EP^{L(j)-\infty}$ is well defined. 
Therefore we can take the $\EP$ version 
of \eqref{e:extterm}, which by the Kunneth formula equals the tensor product of:
\begin{equation}
\begin{aligned}
&\EP_{\GL(t-j)}(\xi(t,j),\sigma_{j,k}),\\
&\EP_{\GL(j)}(\S(\GL(j)),\tau_{j,k})^{\infty}\simeq \tau_{j,k},\\
&\EP_{\Sp(2m-2t)}(\omega_{m-t,n-j},\pi_2).
\end{aligned}
\end{equation}
The first term is $0$ unless $j=t$, in which case $N(t,j)$ is trivial, 
and \eqref{e:sum} is simply $\pi_1$. 
Here is the conclusion.

\begin{theorem}
Suppose $\O(N)$ is the isometry group of an orthogonal space of dimension $N$ and Witt index $n$. 
Consider the oscillator representation  $\omega_{m,n}$  for the dual pair
$(\Sp(2m),\O(N))$. 
Fix $0\le t\le \min(m,n)$ and let $P(t)=M(t)N(t)$ be the
stabilizer in $\Sp(2m)$ of an isotropic subspace of dimension $t$.
Fix   an irreducible representation $\pi$  of $M(t)\simeq \GL(t)\times \Sp(2m-2t)$.
Consider the space
$$
\EP_{\Sp(2m)}(\omega_{m,n},i_{P(t)}^{\Sp(2m)}(\pi))^\infty
$$
where the smooth vectors are with respect to $\O(N)$. This is an element of the Grothendieck group 
of finite length $\O(N)$-modules.

If $t>n$, then 
$$
\EP_{\Sp(2m)}(\omega_{m,n},i_{P(t)}^{\Sp(2m)}(\pi))^\infty=0.
$$

Suppose $t\le n$, and let $Q(t)=L(t)U(t)$ be the stabilizer of a $t$-dimensional
isotropic subspace of the orthogonal space, so 
$$
L(t)\simeq \GL(t)\times \O(N-2t).
$$
Write
$\pi=\pi_1\bx\pi_2$ for  $M(t)=\GL(t)\times\Sp(2m-2t)$.
Then
$$
\EP_{\Sp(2m)}(\omega_{m,n},
i_{P(t)}^{\Sp(2m)}(\pi))^{\infty}\simeq i_{Q(t)}^{\O(n)}(\pi_1\bx\EP_{\Sp(2m-2t)}(\omega_{m-t,n-t},\pi_2)^{\infty}).
$$
\end{theorem}

This can be stated more succinctly as follows.
Let $\omega_{M(t),M'(t)}$ be the oscillator representation  for the 
dual pair $(M(t),M'(t))=(\GL(t)\times \Sp(2m-2t),\GL(t)\times \O(N-2t))$. 

\begin{corollary}
Consider the oscillator representation  $\omega_{G,G'}$ of the dual pair $(G,G')=(\Sp(2m),\O(N))$.
For $t\le m$, let $P(t)=M(t)N(t)\subset \Sp(2m)$ be the stabilizer of an isotropic space of dimension $t$.
Similarly,
if $t\le n$, let $Q(t)=L(t)U(t)\subset \O(N)$ be the stabilizer of an isotropic space of dimension $t$.

If $t\le\min(m,n)$, 
let $\omega_{M(t),L(t)}$ be the oscillator representation  for the dual pair
$(M(t),L(t))=(\GL(t)\times \Sp(2m-2t),\GL(t)\times \O(N-2t))$.

Fix an irreducible representation $\pi$ of $M(t)$.
Then
$$
\EP_{G}(\omega_{G,G'},i_{P}^{G}(\pi))^{\infty}
\simeq 
\begin{cases}
0&t>n\\
i_{Q(t)}^{G'}(\EP_{M(t)}(\omega_{M(t),L(t)},\pi)^{\infty})&t\le n.
\end{cases}
$$
\end{corollary}

\bibliographystyle{plain}

\begin{thebibliography}{10}

\bibitem{dummies}
Theta correspondence for dummies.
\newblock http://math.mit.edu/conferences/howe/adams.php.

\bibitem{blanc}
Philippe Blanc.
\newblock Projectifs dans la cat\'egorie des {$G$}-modules topologiques.
\newblock {\em C. R. Acad. Sci. Paris S\'er. A-B}, 289(3):A161--A163, 1979.

\bibitem{casselman_padic}
W.~Casselman.
\newblock Introduction to the theory of admissible representations of $p$-adic
  reductive groups.
\newblock 1995.
\newblock preprint.

\bibitem{gan_takeda_howe_conjecture}
Wee~Teck Gan and Shuichiro Takeda.
\newblock A proof of the Howe duality conjecture.
\newblock {\em J. Amer. Math. Soc.}, 29(2):473--493, 2016.

\bibitem{howe_transcending}
Roger Howe.
\newblock Transcending classical invariant theory.
\newblock {\em J. Amer. Math. Soc.}, 2(3):535--552, 1989.

\bibitem{rumelhart_bernstein}
Karl~Rumelhart (Notes for a course of Joseph~Bernstein).
\newblock Draft of: Representation of $p$-adic groups.
\newblock 1992.


\bibitem{kudla_induction}
Stephen~S. Kudla.
\newblock On the local theta-correspondence.
\newblock {\em Invent. Math.}, 83(2):229--255, 1986.

\bibitem{minguez_type_II}
Alberto M{\'{\i}}nguez.
\newblock Correspondance de {H}owe explicite: paires duales de type {II}.
\newblock {\em Ann. Sci. \'Ec. Norm. Sup\'er. (4)}, 41(5):717--741, 2008.

\bibitem{mvw}
Colette M{\oe}glin, Marie-France Vign{\'e}ras, and Jean-Loup Waldspurger.
\newblock {\em Correspondances de {H}owe sur un corps {$p$}-adique}, volume
  1291 of {\em Lecture Notes in Mathematics}.
\newblock Springer-Verlag, Berlin, 1987.

\bibitem{prasadext}
D.~Prasad.
\newblock Ext-analogues of branching laws.
\newblock preprint, arXiv:1306.2729.

\bibitem{waldspurger_howe_conjecture}
J.-L. Waldspurger.
\newblock D\'emonstration d'une conjecture de dualit\'e de {H}owe dans le cas
  {$p$}-adique, {$p\neq 2$}.
\newblock In {\em Festschrift in honor of {I}. {I}. {P}iatetski-{S}hapiro on
  the occasion of his sixtieth birthday, {P}art {I} ({R}amat {A}viv, 1989)},
  volume~2 of {\em Israel Math. Conf. Proc.}, pages 267--324. Weizmann,
  Jerusalem, 1990.


\end{thebibliography}
\def\cprime{$'$} \def\cftil#1{\ifmmode\setbox7\hbox{$\accent"5E#1$}\else
  \setbox7\hbox{\accent"5E#1}\penalty 10000\relax\fi\raise 1\ht7
  \hbox{\lower1.15ex\hbox to 1\wd7{\hss\accent"7E\hss}}\penalty 10000
  \hskip-1\wd7\penalty 10000\box7}
  \def\cftil#1{\ifmmode\setbox7\hbox{$\accent"5E#1$}\else
  \setbox7\hbox{\accent"5E#1}\penalty 10000\relax\fi\raise 1\ht7
  \hbox{\lower1.15ex\hbox to 1\wd7{\hss\accent"7E\hss}}\penalty 10000
  \hskip-1\wd7\penalty 10000\box7}
  \def\cftil#1{\ifmmode\setbox7\hbox{$\accent"5E#1$}\else
  \setbox7\hbox{\accent"5E#1}\penalty 10000\relax\fi\raise 1\ht7
  \hbox{\lower1.15ex\hbox to 1\wd7{\hss\accent"7E\hss}}\penalty 10000
  \hskip-1\wd7\penalty 10000\box7}
  \def\cftil#1{\ifmmode\setbox7\hbox{$\accent"5E#1$}\else
  \setbox7\hbox{\accent"5E#1}\penalty 10000\relax\fi\raise 1\ht7
  \hbox{\lower1.15ex\hbox to 1\wd7{\hss\accent"7E\hss}}\penalty 10000
  \hskip-1\wd7\penalty 10000\box7} \def\cprime{$'$} \def\cprime{$'$}
  \def\cprime{$'$} \def\cprime{$'$} \def\cprime{$'$} \def\cprime{$'$}
  \def\cprime{$'$} \def\cprime{$'$}

\end{document}